\documentclass[11pt,letterpaper,english]{amsart}

\usepackage{times,newtxmath,enumerate,dsfont} 
\usepackage[text={6.5in,8.2in},centering]{geometry} 
  
\newcommand{\inv}{^{\raisebox{.2ex}{$\scriptscriptstyle-1$}}}   
\newcommand{\add}[1]{\textcolor{Red}{#1}}

\usepackage[dvipsnames]{xcolor}   
\usepackage{xparse}
\usepackage{xr-hyper}
\usepackage[linktocpage=true,colorlinks=true,hyperindex,citecolor=RoyalBlue,linkcolor=Royal Blue]{hyperref}   

\newtheorem{theorem}{Theorem}[section]
\newtheorem{proposition}[theorem]{Proposition}
\newtheorem{lemma}[theorem]{Lemma}
\newtheorem{corollary}[theorem]{Corollary}
\theoremstyle{definition}
\newtheorem{definition}[theorem]{Definition}

\newtheorem{cremark}[theorem]{Concluding Remarks}
\numberwithin{equation}{section}

\begin{document}
	
\author{Themba Dube}

\address{\textsc{T. Dube}: [1] Department of Mathematical Sciences, University of South Africa,
South Africa; [2] National Institute for Theoretical and Computational Sciences (NITheCS), South Africa.}

\email{dubeta@unisa.ac.za}

\author{Amartya Goswami}
\address{\textsc{A. Goswami}:
[1] Department of Mathematics and Applied Mathematics, University of Johannesburg, South Africa; [2] National Institute for Theoretical and Computational Sciences (NITheCS), South Africa.}
\email{agoswami@uj.ac.za}

\title{Some aspects of $k$-ideals of semirings}

\date{}

\subjclass{16Y60.}


\keywords{semiring, $k$-ideal, $k$-closure operation, weakly Noetherian semiring.}

\begin{abstract}
The aim of this paper is to study some distinguished classes of $k$-ideals of semirings, which include $k$-prime, $k$-semiprime, $k$-radical, $k$-irreducible, and $k$-strongly irreducible ideals. We discuss some of  the properties of $k$-ideals and their products, intersections, and ideal quotients under semiring homomorphisms.
\end{abstract}
\maketitle

\section{Introduction}

As algebraic structures, semirings are definitely a natural choice of generalizations of rings, and it is appropriate to ask which properties of rings can be extended to semirings. One may expect semirings always to be extended to rings, but \cite{V39} gave examples of semirings that cannot be embedded in rings. In studying the structures of rings, just as ideals play crucial roles, the same is true for semirings. However, the ideals of semirings are significantly different in nature than those of rings. The lack of subtraction in semirings shows that many results in rings have no counterparts in semirings. The notion of $k$-ideal introduced in \cite{H58} is an attempt to compensate for this lack of subtractivity. Since the introduction of $k$-ideals, several studies have been made (see \cite{AA08, G99, H15, JRT22, L95, L65, L15, SA92, SA93, WSA96}). 

In this paper, we shall touch upon some aspects of $k$-ideals that have not been addressed in the literature. The aim of this and the sequel paper, \cite{DG23} is to contribute further on the theory of $k$-ideals by using results from the above-mentioned works and extending results from rings and semirings. To obtain this, we have lifted results from
\cite{G99, I56} on  ideals of semirings, and from \cite{L01, B72, AM69, A08, B53} on ideals of rings.

The result that plays a key role in studying distinguished classes (call one such as X) of $k$-ideals is the following ``exchange principle'':
\begin{equation}\label{kxkx}
k\text{-X ideal}= k\text{-ideal + X ideal.} 
\end{equation}

For example, to prove the properties of $k$-prime ideals (see Definition \ref{kpri}), it will be sufficient to prove them for prime $k$-ideals. It turns out that all the distinguished classes of $k$-ideals that we have considered here satisfy the above exchange principle. The other tool that has been extensively used in the proofs is the alternative formulation of $k$-ideals in terms of $k$-closure operations (see Definition \ref{adki}).

We shall now briefly describe the content of the paper. In \S \ref{kid}, we recall the definition of the $k$-closure operation and gather some of its properties. We have definition of $k$-ideals in terms of the $k$-closure operations and study some operation on $k$-ideals. The purpose of \S\ref{scki} is to study some distinguished classes of $k$-ideals, which contain $k$-prime, $k$-semiprime, and $k$-radical ideals in \S\ref{ptki}, whereas in \S\ref{itki}, we discuss $k$-irreducible and $k$-strongly irreducible ideals. Finally, in \S\ref{kekc}, we introduce the notions of $k$-contractions and $k$-extensions of $k$-ideals. We conclude this paper with remarks on some further works.

\section{\large{$k$-ideals}}\label{kid}  
A (commutative) \emph{semiring} is a system $(R  ,+,0,\cdot, 1)$ such that $(R  ,+,0)$ is a commutative monoid, $(R  , \cdot,1)$ is a commutative monoid, $0\cdot x=0=x\cdot 0$ for all $x\in R  ,$ and $\cdot$ distributes over $+$. We shall write $x y$ for $x\cdot y.$ By a semiring, in this paper, we mean a commutative semiring with multiplicative identity element $1$.
A \emph{semiring homomorphism} $\phi\colon R  \to R '$ is a map such that $\phi(x+y)=\phi(x)+\phi(y),$ $\phi(xy)=\phi(x)\phi(y),$ $\phi(0)=0$, and $\phi(1)=1$ for all $x,$ $y\in R .$ 
An \emph{ideal} $I $ of a semiring $R $ is an additive submonoid of $R $ such that 
$rx\in I $
for all $x\in I $ and $r\in R  .$ An ideal $I$ is called \emph{proper} if $I\neq R .$ We also use the symbol $0$ to denote the zero ideal of a semiring $R $. If $S$ is a nonempty subset of a semiring $R$, then $\langle S \rangle$ will denote the ideal generated by $S$. 

\begin{definition}
[\cite{H58}] A \emph{$k$-ideal} (or \emph{subtractive ideal}) $I$ of a semiring $R $ is an ideal of $R $ such that $x,$ $x+y\in I$ implies $y\in I$\footnote{Note the typo in \cite{H58}:  In the definition of a $k$-ideal, `$y\in S$' should be `$y\in I$'.}.
\end{definition}

Surely, the zero ideal is a $k$-ideal and is contained in every $k$-ideal of a semiring $R $. We denote by $\mathcal{I}\!(R) $ (respectively by $\mathcal{I}_k(R) $) the set of all ideals (respectively all $k$-ideals) of $R $. A \emph{$k$-proper} ideal is a proper $k$-ideal of $R$.
A $k$-ideal of a semiring can also be defined in terms of a closure operation on $\mathcal{I}\!(R) $. 
\begin{definition}
[\cite{SA92}]
Let $R $ be a semiring.  The \emph{$k$-closure} operation on $\mathcal{I}(R) $  is defined by
\begin{equation}
\label{clkdef}
\mathcal{C}_k(I)=\{r\in R\mid r+x\in I\;\text{for some}\; x\in I\}.
\end{equation} 	
\end{definition}

The next lemma gathers the properties of closure operations needed in the sequel.

\begin{lemma}\label{lclk}
In the following, $I$, $\{I_{\lambda}\}_{\lambda \in \Lambda}$, and $J$ are ideals of a semiring $R $. 
\begin{enumerate}[\upshape (1)]
		
\item\label{iclk} $\mathcal{C}_k(I)$ is the smallest $k$-ideal containing $I$.
		
\item 
$\mathcal{C}_k(0)=0.$
		
\item \label{ckr}
$ \mathcal{C}_k(R) =R .$
		
\item\label{clcl} $\mathcal{C}_k(\mathcal{C}_k(I))=\mathcal{C}_k(I).$
		
\item\label{ijcl} $I\subseteq J$ implies $\mathcal{C}_k(I)\subseteq \mathcal{C}_k(J).$
		
\item \label{clu}
$\mathcal{C}_k(\langle I\cup J\rangle )\supseteq \mathcal{C}_k(I) \cup \mathcal{C}_k(J).$
		
\item\label{arbin} $\mathcal{C}_k\left(\bigcap_{\lambda\in \Lambda}I_{\lambda}\right)=\bigcap_{\lambda\in \Lambda} \mathcal{C}_k (I_{\lambda}).$
		
\item\label{altd} $I$  is a $k$-ideal if and only if $I=\mathcal{C}_k(I).$
		
\item\label{clijk} $\mathcal{C}_k(IJ)\supseteq \mathcal{C}_k(I)\,\mathcal{C}_k(J).$
\end{enumerate}
\end{lemma}

\begin{proof}
(1) \cite[Proposition 3.1]{JRT22}. 

(2)--(7) Straightforward. 
	
(8)  \cite[Lemma 2.2]{SA92}. 
	
(9) Let $r\in \mathcal{C}_k(I)\,\mathcal{C}_k(J).$ Then there exist $r'\in \mathcal{C}_k(I)$ and $r''\in \mathcal{C}_k(J)$ such that $r'+i\in I$, $r''+j\in J$, and $r=r'r''$, where $i\in I$ and $j\in J.$ Notice that
\[r'r''+r'j+r''i+ij=(r'+i)(r''+j)\in IJ.\]
Since $r'r''+r'j+r''i\in R $, $r'j+r''i\in R ,$ and $R $ is a $k$-ideal, $r=r'r''\in R $, and this completes the proof
\end{proof}

Thanks to Lemma \ref{lclk}(\ref{altd}), we now have an alternative definition of a $k$-ideal.

\begin{definition}\label{adki}
An ideal $I$ of a semiring $R $ is called a \emph{$k$-ideal} if $\mathcal{C}_k(I)=I.$
\end{definition}

From Lemma \ref{lclk}(\ref{iclk}),  it follows that the $k$-closure is indeed a map
$\mathcal{C}_k\colon \mathcal{I}\!(R) \to \mathcal{I}_k(R) $
defined by (\ref{clkdef}). Our next goal is to study various operations on $k$-ideals of a semiring. It is easy to see that if $\{I _{\lambda}\}_{\lambda\in \Lambda}$ is a family of $k$-ideals, then their \emph{intersection} $\bigcap_{\lambda\in \Lambda} I _{\lambda}$ is also a $k$-ideal. However, the sum of two $k$-ideals of a semiring need not be a $k$-ideal\footnote{\cite[Example 6.19]{G99}: $2\mathds{N}$ and $3\mathds{N}$ are $k$-ideals of the semiring $\mathds{N}$, however $2\mathds{N}+3\mathds{N}=\mathds{N}\setminus \{1\}$ is not a $k$-ideal of $\mathds{N}$.}, but it is so in a lattice ordered semiring (\textit{cf}.
\cite[Corollary 21.22]{G99}). If $I$ and $J$ are $k$-ideals of a semiring  $R$, define  their \emph{product} as $IJ=\mathcal{C}_k(\langle \{xy\mid x\in I, y\in J\}\rangle)$. The following lemma gives us the expected relation between products and intersections of $k$-ideals. The proof is straightforward.

\begin{lemma}\label{psi}
For any two $k$-ideals $I$ and $J$ of a semiring  $R$, we have $IJ\subseteq I\cap J.$
\end{lemma}



If $X$ is a subset of a semiring $R $, then the \emph{annihilator} of $X$ is defined by \[\mathrm{Ann}_R (X)=\{r\in R\mid rx=0\;\text{for all}\; x\in X\}.\] The following result is proved in \cite[Example 6.10]{G99}.

\begin{proposition}
If $X$ is a nonempty subset of a semiring $R $ and $X\neq \{0\},$ then  $\mathrm{Ann}_R (X)$ is a $k$-ideal.
\end{proposition}

Suppose that $I$ and $J$ are ideals of a semiring $R $.  Then the \emph{ideal quotient} of $I$ over $J$ is defined by \[(I : J)=\{r\in R \mid rJ\subseteq I\}.\]
The following proposition and its corollary give examples of $k$-ideals constructed from ideal quotients.

\begin{proposition}\label{icj}
If $I$ is a $k$-ideal and $J$ is an ideal of a semiring $R $, then $(I:J)$ is a $k$-ideal of $R $.
\end{proposition}

\begin{proof}
Let $r,$ $r+r'\in (I\colon J)$. This implies $rJ\subseteq I$ and $rJ+r'J=(r+r')J\subseteq I$. Since $I$ is a $k$-ideal, we have $r'J\subseteq I$, whence $r'\in (I\colon J).$
\end{proof}

\begin{corollary}
Suppose that  $J$, $\{J_{\lambda}\}_{\lambda \in \Lambda}$, $K$ are ideals and  $I$, $\{I_{\omega}\}_{\omega \in \Omega}$ are $k$-ideals of a semiring $R $. Then  $ (I\colon J),$  $((I\colon J)\colon K),$ $(I\colon JK),$ $((I\colon K)\colon J),$ $ (\bigcap_{\omega} I_{\omega}\colon J),$ $ \bigcap_{\omega}(I_{\omega}\colon J),$ $ (I\colon \sum_{\lambda}J_{\lambda}),$ and $ \bigcap_{\lambda}(I\colon J_{\lambda})$ are all $k$-ideals of $R $.
\end{corollary}

The lattice of all ideals of a ring is modular, whereas  the same is not true for a semiring. However, we have the following result announced in \cite{H58}. A proof of it can be found in \cite[Proposition 6.38]{G99}.

\begin{proposition}
For every semiring $R $, $\mathcal{I}_k(R) $ is a modular lattice.
\end{proposition}

\section{Some classes of $k$-ideals}\label{scki}

The purpose of this section is to study some distinguished classes of $k$-ideals of semirings. These classes of $k$-ideals are obtained by ``restricting'' the usual definitions of the corresponding classes of ideals to $k$-ideals.

\begin{definition}
A proper $k$-ideal of a semiring  $R $ is called \emph{$k$-maximal} if it is not properly contained in another proper $k$-ideal. 	
\end{definition}

The following result from \cite[Corollary 2.2]{SA93} guarantees that the set of $k$-maximal ideals in a semiring is nonempty.

\begin{lemma}\label{exma} 
Every proper $k$-ideal of a semiring  $R $ is contained in a $k$-maximal ideal of $R $.
\end{lemma}

As it has been pointed out in the introduction that the exchange principle (\ref{kxkx}) holds for the classes of $k$-ideals that we study, the following proposition is the very first example of it for $k$-maximal ideals.

\begin{proposition}
\label{eqsm}
An ideal $M$ of a semiring $R $ is $k$-maximal if and only if it is a $k$-ideal and a maximal ideal of $R $.
\end{proposition}

\begin{proof}
It is obvious that every $k$-ideal of $R $ which is also a maximal ideal is  $k$-maximal. Suppose that  $M$ is a $k$-maximal ideal and $M\subsetneq I\subsetneq R $ for some $I\in \mathcal{I}\!(R)\setminus \mathcal{I}_k\!(R) .$ This implies that
\[M=\mathcal{C}_k(M)\subsetneq I \subsetneq \mathcal{C}_k(I)\subsetneq \mathcal{C}_k(R) =R,
\]
where the first and last equalities follow respectively from (\ref{altd}) and (\ref{ckr}) of Lemma \ref{lclk}, and the first and the last strict  inclusions follow from Lemma \ref{lclk}(\ref{ijcl}). The middle strict inclusion follows from the hypothesis $I\in \mathcal{I}\!(R)\setminus \mathcal{I}_k\!(R)$ and Lemma \ref{lclk}(\ref{altd}).
Since $M$ is $k$-maximal, $M \subsetneq \mathcal{C}_k(I)\subsetneq R$ leads to a contradiction.
\end{proof}

\subsection{$k$-prime and $k$-semiprime ideals}\label{ptki}

In noncommutative semirings, the notions of prime and semiprime ideals play significant roles. In this section, we shall extend those types of ideals  to $k$-prime and $k$-semiprime ideals. The notion of a radical of an ideal will be replaced by $k$-radicals.

\begin{definition}\label{kpri}
A $k$-proper ideal $P$ of a semiring $R $ is called \emph{$k$-prime}  if  $IJ\subseteq P$ implies either $I\subseteq P$ or $J\subseteq P$ for all $I, J\in \mathcal{I}_k(R).$ We denote by $\mathrm{spec}_k(R)$ the set of all $k$-prime ideals of $R$ 
\end{definition}

The exchange principle (\ref{kxkx}) for $k$-prime ideals is proved in   \cite[Proposition 3.5]{JRT22}.

\begin{proposition}
\label{eqp} 
An ideal $P$ of a semiring $R$ is $k$-prime if and only if $P$ is a prime ideal as well as a $k$-ideal of $R$.
\end{proposition}

As in rings (with identity), we also have the following result for $\mathrm{spec}_k(R)$.

\begin{proposition}
\label{kmkp} Every nonzero semiring contains a minimal $k$-prime ideal.
\end{proposition}

\begin{proof}
Suppose that  $R$ is a nonzero semiring. It follows from Lemma \ref{exma} that $R $ has a $k$-maximal ideal, which by \cite[Theorem 3.4 ]{H15} is also a prime $k$-ideal and hence is a $k$-prime ideal by \cite[Proposition 3.5]{JRT22}. Therefore, the set $\mathrm{spec}_k(R)$ is nonempty. The claim now follows from a routine application of Zorn's lemma.	
\end{proof}

\begin{proposition}\label{kpp}
\label{eei} For a $k$-prime ideal $P$  of a semiring $R$, the following are equivalent.
\begin{enumerate}[\upshape (1)]
		
\item $xy\in P$ implies either $x\in P$ or $y\in P$ for all $x, y\in $.
		
\item $IJ\subseteq P$ implies either $I\subseteq P$ or $J\subseteq P$ for all $I, J \in \mathcal{I}_k(R)$.
\end{enumerate}
\end{proposition}

\begin{proof}
Let $I\nsubseteq P$ and $j\in J$. If $i\in I\setminus P$, then $ij\in IJ\subseteq P$. But $i\notin P$. By hypothesis, this implies that $j\in P$, or in other words, $J\subseteq P.$ Conversely, suppose that  $xy\in P$ and $x\notin P.$ From \cite[Proposition 3.5]{JRT22}, it follows that $\langle xy\rangle_k =\langle x\rangle_k \langle y \rangle_k$, where $\langle x\rangle_k$ is the smallest $k$-ideal of $R $ containing $x$. Therefore, $\langle x\rangle_k \langle y \rangle_k\subseteq P$ and $\langle x\rangle_k \nsubseteq P$. Since  $P$ is  $k$-prime, we must have $\langle y \rangle_k \subseteq P$, whence $y\in P.$ 
\end{proof}

The following proposition is lifted from rings, and the proof is identical to the ring-theoretic version of it, which can be found in \cite[Lemma 3.19]{AK13}.

\begin{proposition}[Prime avoidance lemma]
Let $I$ be a subset of a semiring $R $ that is stable under addition and multiplication, and $P_1, \ldots, P_n$ be $k$-ideals such that $P_3, \ldots, P_n$
are $k$-prime ideals. If $I\nsubseteq P_j$ for all $j,$ then there is an $x \in I$ such that $x \notin
P_j$ for all $j$.
\end{proposition}

Recall from	\cite[Definition 2.5]{L15} that a semiring $R$ is called \emph{weakly Noetherian} if every ascending chain of $k$-ideals of
$R $ is ultimately stationary.  If $R $ is such that every ideal of it is a $k$-ideal, then  \cite[Proposition 7.17]{G99} shows
$R$ is weakly Noetherian if and only if every $k$-prime ideal of a semiring $R $ is finitely generated.	Moreover, in a weakly Noetherian semiring $R$, \cite[Corollary 6.6]{L15} says that the radical $\mathcal{R}_k(I)$ of a $k$-ideal $I$ can be represented as a finite intersection of $k$-prime ideals. In a weakly Noetherian semiring, we have another result, but for minimal $k$-prime ideals.

\begin{proposition}
\label{wnknp} 
If $R $ is a weakly Noetherian semiring, then the set of minimal $k$-prime ideals of a semiring $R $ is finite.
\end{proposition}

\begin{proof}
We give a topological proof. If $R $ is weakly Noetherian, then the topological space (endowed with Zariski topology) $\mathrm{spec}_k(R) $ is also Noetherian, and thus $\mathrm{spec}_k(R)$ has  finitely many irreducible
components. Now every irreducible closed subset of $\mathrm{spec}_k(R) $
is of the form \[\mathrm{V}(P)=\{Q\in \mathrm{spec}_k(R) \mid P\subseteq Q\},\] where $P$ is a minimal $k$-prime ideal. Thus $\mathrm{V}(P)$ is irreducible component if and only if $P$ is a minimal $k$-prime ideal. Hence, the number of minimal $k$-prime ideals of $R $ is finite.
\end{proof}

We shall now define the notion of  $k$-semiprime ideals. 
It will be shown in Theorem \ref{spkr} that $k$-semiprime ideals are equivalent to $k$-radical ideals (see Definition \ref{drad}).

\begin{definition}\label{kspr}
A $k$-proper ideal $P$ of a semiring $R $ is called \emph{$k$-semiprime} if $I^2\subseteq P$ implies $I\subseteq P,$ for all $k$-ideals $I$ of $R $. 
\end{definition}

Similar to $k$-prime ideals, we also have exchange principle for $k$-semiprime ideals.

\begin{proposition}\label{eqpp}
An ideal $Q$ of a semiring $R $ is $k$-semiprime if and only if $Q$ is a $k$-ideal and a semiprime ideal of $R $.
\end{proposition}

\begin{proof}
Observe that if $Q$ is a $k$-ideal and also a semiprime ideal of a semiring $R $, then by Definition \ref{kspr}, $Q$ is $k$-semiprime. Conversely, let $Q$ be a $k$-semiprime ideal and $I^2\subseteq Q$ for some $I\in \mathcal{I}\!(R) .$ Then
\[(\mathcal{C}_k(I))(\mathcal{C}_k(I))\subseteq \mathcal{C}_k(I^2)\subseteq \mathcal{C}_k(Q)=Q,\]
where the first and second inclusions respectively follow from Lemma \ref{lclk}(\ref{clijk}) and Lemma \ref{lclk}(\ref{ijcl}). Since $Q$ is $k$-semiprime, applying  Lemma \ref{lclk}(\ref{iclk}), we obtain that $I\subseteq Q.$ 
\end{proof}

In rings, it is well known that the definition of the radical of an ideal $I$ can be formulated as the intersection of all prime ideals that contain $I$. If we restrict $I$ to being a $k$-ideal and prime ideals to being $k$-primes, then we obtain the notion of a $k$-radical.

\begin{definition}[\cite{JRT22}]\label{drad}
	
The \emph{$k$-radical of a $k$-ideal $I$} of a semiring $R$ is defined by
\begin{equation}\label{radki}
\mathcal{R}_k(I)=\bigcap_{I\subseteq P}\{P\mid P\in \mathrm{spec}_k(R) \}.
\end{equation}
If $\mathcal{R}_k(I)=I,$ then $I$ is called a \emph{$k$-radical ideal}.	
\end{definition}
In the following lemma we gather some elementary properties of $k$-radicals. 

\begin{lemma}\label{radpr}
In the following, $I$ and $J$ are $k$-ideals of of a semiring   $R$.
\begin{enumerate}[\upshape (1)]
		
\item\label{irad} $\mathcal{R}_k(I)$ is a $k$-ideal containing $I$.

\item $\mathcal{R}_k(\mathcal{R}_k(I))=\mathcal{R}_k(I).$
		
\item\label{ijicj} $\mathcal{R}_k(IJ)= \mathcal{R}_k(I\cap J)= \mathcal{R}_k(I)\cap \mathcal{R}_k(J).$
		
\item\label{irir} $\mathcal{R}_k(I)=R $ if and only if $I=R .$
\end{enumerate}
\end{lemma}

\begin{proof}
(1) The fact that $I\subseteq \mathcal{R}_k(I)$ follows from Definition \ref{drad}. To show $\mathcal{R}_k(I)$ is a $k$-ideal, let $x, x+y\in \mathcal{R}_k(I).$ This implies $x, x+y\in P$ for all $k$-prime ideals that contain $I$. Since each such $P$ is a $k$-ideal, $y\in P$. Hence $y\in \mathcal{R}_k(I).$

(2) Since by (1), $\mathcal{R}_k(I)$ is a $k$-ideal, applying Lemma \ref{lclk}(\ref{iclk}) gives $\mathcal{R}_k(I)\subseteq \mathcal{R}_k(\mathcal{R}_k(I)).$ Since by (1), $I\subseteq \mathcal{R}_k(I),$ by Definition \ref{drad}, we have $\mathcal{R}_k(I)\supseteq \mathcal{R}_k(\mathcal{R}_k(I)).$
	
(3) Follows from Lemma \ref{lclk}(\ref{ijcl}) and Lemma \ref{psi}.
	
(4) Follows from Definition \ref{drad}. 
\end{proof}


Next, we wish to prove the equivalence between $k$-semiprime ideals and $k$-radical ideals of a semiring. This equivalence is well known to hold between semiprime ideals and radical ideals of (noncommutative) rings and semirings.  In the noncommutative case, we require the notions of $m$-systems and $n$-systems of rings and semirings, whereas in the context of $k$-ideals of commutative semirings, only multiplicatively closed subsets are sufficient. To obtain the equivalence, we first proceed through a series of lemmas.

\begin{lemma}\label{rpms}
A $k$-ideal $P$ of a semiring $R$ is $k$-prime if and only if $R\setminus P$ is a multiplicatively closed subset of $R$.
\end{lemma}

\begin{proof}
Suppose that  $P$ is $k$-prime and $x,$ $y\in R\setminus P$. Then by Proposition \ref{kpp}, $xy\notin P$, and hence $xy\in R\setminus P.$ Conversely, let $x,$ $y\notin P$. This implies that $x,$ $y\in R\setminus P$. Since $R\setminus P$ is multiplicatively closed, $xy\in R\setminus P$, and hence $xy\notin P$, proving that $P$ is prime. Therefore, by Proposition \ref{kpp}, $P$ is $k$-prime. 
\end{proof}

\begin{lemma}\label{mxkp}
Let $S$ be a multiplicatively closed subset of a semiring $R$. Suppose that  $P$ is a $k$-ideal which is maximal with respect to the property that $P\cap S=\emptyset$. Then $P$ is $k$-prime.
\end{lemma}

\begin{proof}
Let $x\notin P$ and $y\notin P$, but $\langle x\rangle \langle y \rangle \subseteq P.$ By the hypothesis on $P$, there exist $s,$ $s'\in S$ such that $s\in \langle x\rangle+P$ and $s'\in \langle y\rangle+P$. This implies that
\[ss'\in (\langle x\rangle+P)(\langle y\rangle+P)=\langle x\rangle \langle y \rangle+P\subseteq P,\]
a contradiction. Hence $P$ is prime and by Proposition \ref{kpp}, $P$ is a $k$-prime ideal.
\end{proof}

\begin{lemma}\label{rkt}
The $k$-radical $\mathcal{R}_k(I)$ of a $k$-ideal $I$ is equal to the set \[T=\{r\in R\mid \mathrm{every\; multiplicatively\; closed\; subset\; containing}\;r\;\mathrm{intersects}\;I\}.\]
\end{lemma}

\begin{proof}
Suppose that  $r\in T$ and $P\in \mathrm{spec}_k(R)$ such that $I\subseteq P$. Then by Lemma \ref{rpms}, $R\setminus P$ is a multiplicatively closed subset of $R$ and $r\notin  R\setminus P.$ Hence $r\in P$. Conversely, let $r\notin T$. This implies that there exists a multiplicatively closed subset $S$ of $R$ such that $r\in S$ and $S\cap I=\emptyset.$ By Zorn's lemma, there exists a $k$-ideal $P$ containing $I$ and maximal with respect to the property that $P\cap S=\emptyset$. By Lemma \ref{mxkp}, $P$ is a prime ideal and by Proposition \ref{kpp}, $P$ is a $k$-prime ideal with $r\notin P$.
\end{proof}

\begin{lemma}\label{mms}
Let $I$ be a $k$-semiprime ideal of a semiring $R$ and $x\in R\setminus I$. Then there exists a multiplicatively closed subset $S$ of $R$ such that $x\in S\subseteq R\setminus I.$
\end{lemma}

\begin{proof}
Define the elements of $S=\{x_1, x_2, \ldots, x_n, \ldots\}$ inductively as follows:
$
x_1:= x;$
$x_2:=x_1x_1;$
$\ldots$;
$x_n:=x_{n-1}x_{n-1};$ $\ldots$. Obviously $x\in S$ and it is also easy to see that $x_i,$ $x_j\in S$ implies that $x_ix_j\in S$.
\end{proof}

\begin{theorem}\label{spkr}
For any $k$-ideal $I$ of a semiring $R $, the following are equivalent.	\begin{enumerate}[\upshape (1)]
	
\item\label{iksp} $I$ is $k$-semiprime.
		
\item\label{iinp} $I$ is an intersection of $k$-prime ideals of  $R $.
		
\item\label{iradi} $I$ is $k$-radical.
\end{enumerate}
\end{theorem}

\begin{proof}
From Definition \ref{drad}, it follows that (3) $\Rightarrow$ (2). Since the intersection of $k$-prime ideals is a $k$-prime ideal and every $k$-prime ideal is $k$-semiprime, (2) $\Rightarrow$ (1) follows. What remains is to show that (1) $\Rightarrow$ (3) and for that, it is sufficient to show  $\mathcal{R}_k(I)\subseteq I.$ Suppose that  $x\notin I$. Then $x\in R\setminus I$ and by Lemma \ref{mms}, there exists a multiplicatively closed subset $S$ of $R$ such that $x\in S\subseteq R\setminus I.$ But $S\cap I=\emptyset$ and hence by Lemma \ref{rkt}, $x\notin \mathcal{R}_k(I).$
\end{proof}

\begin{corollary}
If $I\in \mathcal{I}_k(R) $, then $\mathcal{R}_k(I)$ is the smallest $k$-semiprime ideal of $R $ containing $I$.
\end{corollary}

\subsection{$k$-irreducible and $k$-strongly irreducible ideals}\label{itki}

Strongly irreducible ideals were introduced in \cite{F49} for commutative rings under the name primitive ideals. In \cite[p.\,301, Exercise 34]{B72}, the same ideals are called quasi-prime. The term ``strongly irreducible'' was first used for noncommutative rings in \cite{B53}. In the context of semirings, a study of these ideals can be found in \cite{G99}. In this section, we introduce the $k$-irreducible and $k$-strongly irreducible ideals of semirings and show some relations with the $k$-prime and $k$-semiprime ideals.

\begin{definition} Let $R $ be a semiring.
\begin{enumerate}[\upshape (1)]
		
\item A $k$-ideal $I$ of $R $ is called \emph{$k$-irreducible} if  $J\cap J'= I$ implies that either $J= I$ or $J'= I$ for all $J, J'\in \mathcal{I}_k(R)$.
		
\item A $k$-ideal $I$ of $R $ is called \emph{$k$-strongly irreducible} if  $J\cap J'\subseteq I$ implies that either $J\subseteq I$ or $J'\subseteq I$ for all $J, J'\in \mathcal{I}_k(R)$.
\end{enumerate}
\end{definition}

It is obvious that every $k$-strongly irreducible ideal is $k$-irreducible, and  it follows from Lemma \ref{psi} that every $k$-prime ideal is $k$-strongly irreducible. We now expect the exchange principle to hold for $k$-irreducible and $k$-strongly irreducible ideals, and here is that result.

\begin{proposition}\label{eqsi}
An ideal $L$ of a semiring $R $ is $k$-irreducible ($k$-strongly irreducible)  if and only if $L$ is   irreducible (strongly irreducible)  as well as a $k$-ideal of  $R $.
\end{proposition}

\begin{proof}
We give a proof for $k$-strongly irreducible ideals, that for $k$-irreducible ideals requiring only a
trivial change of terminology.
Suppose that $L$ is a $k$-strongly irreducible ideal and $I$, $J$ are ideals of $R $ such that $I\cap J\subseteq L.$ This implies
\[ \mathcal{C}_k(I)\cap \mathcal{C}_k(J)=\mathcal{C}_k(I\cap J)\subseteq \mathcal{C}_k(L)=L,\]
where, the first equality follows from Lemma \ref{lclk}(\ref{arbin}) and the inclusion from Lemma \ref{lclk}(\ref{ijcl}). By hypothesis, this implies that either $ \mathcal{C}_k(I)\subseteq L$ or $ \mathcal{C}_k(J)\subseteq L$. Since by Lemma \ref{lclk}(\ref{iclk}), $I\subseteq \mathcal{C}_k(I)$ for all $I\in \mathcal{I}\!(R) ,$ we have the desired claim. The proof of the converse statement is obvious.
\end{proof}

It is known (see \cite[Proposition 7.33]{G99}) that a strongly irreducible ideal of a semiring has the following equivalent `elementwise' property: If $a$, $b \in R$ satisfy $\langle a\rangle \cap \langle b\rangle \subseteq I$, then either $a\in I$ or $b\in I$. The next proposition shows that a similar result holds for $k$-strongly irreducible ideals.

\begin{proposition}
\label{abi} 
$I$ is a $k$-strongly irreducible ideal of a semiring $R $ if and only if	for all $a$, $b\in R $ satisfy $\mathcal{C}_k(\langle a\rangle)  \cap \mathcal{C}_k(\langle b\rangle) \subseteq  I$ implies either $a\in I$ or $b\in I$.
\end{proposition}

\begin{proof}
Suppose that $I$ is a $k$-strongly irreducible ideal of $R $ and $\langle a\rangle  \cap \langle b\rangle \subseteq  I$ for all $a,$ $b\in R .$ This implies
\[\mathcal{C}_k(\langle a\rangle \cap \langle b\rangle)=\mathcal{C}_k(\langle a\rangle)\cap  \mathcal{C}_k(\langle b\rangle) \subseteq \mathcal{C}_k (I)=I,\]
where the first equality follows from Lemma \ref{lclk}(\ref{arbin}).
Since $I$ is  $k$-strongly irreducible, we must have either $a\in \langle a\rangle\subseteq  \mathcal{C}_k(\langle a\rangle) \subseteq I$ or  $b\in \langle b\rangle\subseteq  \mathcal{C}_k(\langle b\rangle) \subseteq I$. To show the converse, suppose that $I$ is a $k$-ideal which is not $k$-strongly irreducible. Then there are $k$-ideals $J$ and $K$ satisfy $J\cap K\subseteq I,$ but $J\nsubseteq I$ and $K\nsubseteq I.$ This implies there exist $j\in J$ and $k\in K$ such that $\mathcal{C}_k(\langle j\rangle)  \cap \mathcal{C}_k(\langle k\rangle) \subseteq  I$, but neither $j\notin I$ nor $k\in I$, a contradiction.
\end{proof}

In Theorem \ref{spkr}, we have seen that a $k$-radical ideal can be expressed as the intersection of $k$-prime ideals containing it. We shall now see that any proper $k$-ideal can be represented in a similar fashion in terms of $k$-irreducible ideals. But first, a lemma.

\begin{lemma}\label{lir}
Suppose that $R $ is a semiring. Let $0\neq x\in R $ and $I$	be a $k$-proper ideal of  $R $ such that $x\notin I.$ Then there exists a $k$-irreducible ideal $J$ of $R $ such that $I\subseteq J$
	and $x\notin J$.	
\end{lemma}

\begin{proof}
Let
$\{J_{\lambda}\}_{\lambda \in \Lambda}$ be a chain of $k$-ideals of $R $ such that $x\notin J_{\lambda}\supseteq I$ for all $\lambda \in \Lambda$. Then
\[x\notin \bigcup_{\lambda \in \Lambda} J_{\lambda}\supseteq I.
\]
By Zorn's lemma, there exists a maximal element $J$ of this chain. Suppose that  $J=J_1\cap J_2.$ By the maximality condition of $J$, we must have $x\in J_1$ and $x\in J_2,$ and hence $x\in J_1\cap J_2=J,$ a contradiction. Therefore, $J$ is the required $k$-irreducible ideal.
\end{proof}

\begin{proposition}
If $I$ is a $k$-proper ideal of a semiring $R $, then $I=\bigcap_{I\subseteq J} J$, where $J$ is a $k$-irreducible ideal of  $R $.
\end{proposition}

\begin{proof}
By Lemma \ref{lir}, there exists a $k$-irreducible ideal $J$ of  $R $ such that $I\subseteq J$. Suppose that  \[J'=\bigcap_{I\subseteq J} \{J\mid J\;\text{is $k$-irreducible}\}.\]
Then $I\subseteq J'$. We claim that $J'=I.$   If $J'\neq I$, then there exists an $x\in J'\setminus I$, and by Lemma \ref{lir}, there exists a $k$-irreducible ideal $J''$ such that $x\notin J''\supseteq I,$ a contradiction.
\end{proof}

It is well known that in a weakly Noetherian semiring,  $k$-radicals have representation as a finite intersection of $k$-prime ideals. In terms of $k$-irreducible ideals, the following proposition extends that result to any $k$-ideal. Note that this proposition also holds for irreducible ideals of Noetherian commutative rings.

\begin{proposition}
Let $R $ be a weakly Noetherian semiring. Then every $k$-ideal of $R $ can be represented as an intersection of a finite number of $k$-irreducible ideals of $R $.
\end{proposition}

\begin{proof}
Suppose that 
\[\mathcal{F}=\left\{J\in  \mathcal{I}_k(R) \mid J\neq \bigcap_{i=1}^n L_i,\;L_i \;\text{is $k$-irreducible}\right\}.\]  It is sufficient to show that $\mathcal{F}=\emptyset$.
Since $R $ is weakly Noetherian, $\mathcal{F}$ has a maximal element $I$. Since $I \in \mathcal{F}$, it is not a finite
intersection of $k$-irreducible ideals of $R $. This implies that $I$ is not $k$-irreducible. Hence, there are $k$-ideals $J$ and $K$ such that  $J\supsetneq I$, $K\supsetneq I$, and $I = J \cap K.$ Since $I$ is
a maximal element of $\mathcal{F}$, we must have $J,$ $K \notin \mathcal{F}.$ Therefore, $J$ and $K$ are a finite intersection of
$k$-irreducible ideals of $R $, which subsequently implies that $I $ is also a finite intersection of $k$-irreducible
ideals of $R $, a contradiction.
\end{proof}

As promised at the beginning of this section, the following result shows relations between prime-type  and irreducible-type $k$-ideals.

\begin{proposition}
Let $R$ be a semiring.
A $k$-ideal $P$ of  $R $ is $k$-prime if and only if it is $k$-semiprime and $k$-strongly irreducible.
		
\end{proposition} 

\begin{proof}
Let $P$ be a $k$-prime ideal of a semiring $R $. Then by Proposition \ref{eqp}, $P$ is a $k$-ideal and a prime ideal of a semiring $R $. This implies that $P$ is $k$-semiprime by Proposition \ref{eqpp}. From Lemma \ref{psi} and Proposition \ref{eqsi}, it follows that $P$ is also $k$-strongly irreducible. Conversely, let $P$ be a $k$-semiprime and $k$-irreducible ideal. Suppose that $I$, $J\in \mathcal{I}_k(R) $ satisfying $IJ\subseteq P$. Then 
\[(I\cap J)^2\subseteq IJ\subseteq P.\]
Since $P$ is $k$-semiprime, this implies that $I\cap J\subseteq P$. But $P$ is also $k$-strongly irreducible, which implies that either $I\subseteq P$ or $J\subseteq P.$\qedhere
	
\end{proof}

For commutative rings, it is known (see \cite[Theorem 2.1(ii)]{A08}) that every proper ideal is contained in a minimal strongly irreducible ideal. Incidentally, the same holds for $k$-strongly irreducible ideals of semirings.

\begin{proposition}
Every $k$-proper ideal of a semiring is contained in a minimal $k$-strongly irreducible ideal.
\end{proposition}

\begin{proof}
Let $I$ be a $k$-proper ideal of a semiring $R $. Consider the chain $(\mathcal{E}, \subseteq)$, where \[\mathcal{E}=\{J\mid I\subseteq J,\;J\;\text{is $k$-strongly irreducible}\}.\]
Since every maximal ideal of a semiring $R $ is strongly irreducible, by Proposition \ref{eqsm} and Proposition \ref{eqsi}, every $k$-maximal ideal is $k$-strongly irreducible, and by Lemma \ref{exma}, every proper $k$-ideal is contained in a $k$-maximal ideal. Therefore the set $\mathcal{E}$ is nonempty. By Zorn's lemma, $\mathcal{E}$ has a minimal element, which is our desired minimal $k$-strongly irreducible ideal.
\end{proof}

The following result shows when all $k$-ideals of a semiring are $k$-strongly irreducible, and its proof is obvious.

\begin{proposition}
Every $k$-ideal of a semiring $R$ is $k$-strongly irreducible if and only if  $\mathcal{I}_k(R) $ is totally ordered. 
\end{proposition}

We conclude this section with a theorem on arithmetical semirings, where $k$-irreducibility and $k$-strongly irreducibility coincide.

\begin{theorem}
In a arithmetic semiring $R $, a $k$-ideal of a semiring $R $ is
$k$-irreducible if and only if it is $k$-strongly irreducible. Conversely, 
if a $k$-irreducible ideal of a semiring $R $ is $k$-strongly
irreducible, then $R $ is arithmetic.
\end{theorem}

\begin{proof}
It has been shown in \cite[Theorem 3]{I56} that in an arithmetic semiring, irreducible and strongly irreducible ideals are equivalent. Thanks to Proposition \ref{eqsi}, we have then our claim. For the converse, \cite[Theorem 7]{I56} says that if irreducibility implies strongly irreducibility, then the semiring is arithmetic. Once again, applying Proposition \ref{eqsi} on this result, we get the converse.
\end{proof}

\begin{corollary}
In an arithmetic semiring, any $k$-ideal is the intersection of all $k$-strongly irreducible ideals containing it.
\end{corollary}

\section{$k$-extensions and $k$-contractions} \label{kekc}

The aim of this final section is to study the properties of $k$-ideals and their products, intersections, and ideal quotients under semiring homomorphisms. These properties are `$k$-idealic' extensions of their (commutative) ring-theoretic versions
(see \cite[Proposition 1.17 and Exercise 1.18]{AM69}).

\begin{definition}
Suppose that  $\phi\colon R \to R '$ is a semiring homomorphism.
\begin{enumerate}[\upshape (1)]
					
\item If $J$ is a $k$-ideal of $R '$, then the \emph{$k$-contraction of} $J$, denoted by $J^c$, is defined by $\phi\inv (J).$
		
\item If $I$ is a $k$-ideal of  $R $, then the \emph{$k$-extension of} $I$, denoted by $I^e$, is defined by $\mathcal{C}_k(\langle \phi(I)\rangle).$ 
\end{enumerate}
\end{definition} 

\begin{theorem}\label{cep}
	
Let $\phi\colon R \to R '$ be a semiring homomorphism. For $k$-ideals $I,$ $I_1$, and $I_2$ of  $R $, and $k$-ideals $J,$ $J_1$, and $J_2$  of $R '$, the following hold.
	
\begin{enumerate}[\upshape (1)]
		
\item\label{jcki}  $J^c$ is a $k$-ideal of  $R $.
		
\item The kernel $\mathrm{ker}\phi$ is a
$k$-ideal of $R .$
		
\item $I^e$ is a $k$-ideal of $R '$.
		
\item\label{ijec} $\mathrm{(a)}\; I\subseteq I^{ec}$ $\mathrm{(b)}\; J\supseteq J^{ce}$ $\mathrm{(c)}\; J^c=J^{cec}$ $\mathrm{(d)}\; I^e=I^{ece}.$
		
\item\label{bij} There is a bijection between the sets $\{I\mid I^{ec}=I\}$ and $\{J\mid J^{ce}=J\}$.
		
\item\label{dri} $\mathrm{(a)}\; (I_1\cap I_2)^e\subseteq I_1^e\cap I_2^e$ $\mathrm{(b)}\; (I_1I_2)^e=I_1^eI_2^e$ $\mathrm{(c)}\; (I_1:I_2)^e\subseteq (I_1^e:I_2^e)$ $\mathrm{(d)}\; (\mathcal{R}_k(I))^e\subseteq \mathcal{R}_k(I)^e.$
		
\item\label{crj} $\mathrm{(a)}\; (J_1\cap J_2)^e= J_1^e\cap J_2^e$ $\mathrm{(b)}\; (J_1J_2)^e\supseteq J_1^eJ_2^e$ $\mathrm{(c)}\; (J_1:J_2)^e\subseteq (J_1^e:J_2^e)$ $\mathrm{(d)}\; (\mathcal{R}_k(J))^e= \mathcal{R}_k(J)^e.$
		
\item\label{rjc} 
$\mathcal{R}_k(J)^c\subseteq \mathcal{R}_k(J^c).$
\end{enumerate}	
\end{theorem}

\begin{proof}
(1)  \cite[Proposition 10.11]{G99}.
	
(2) Since $0$ is a $k$-ideal of $R'$, the claim follows from (1).
	
(3) Follows from Lemma \ref{lclk}(\ref{clcl})
	
(4) For (\ref{ijec}.a), observe that $\phi(I)\subseteq \mathcal{C}_k(\langle \phi(I)\rangle),$ whence \[I\subseteq \phi(I)^c\subseteq \mathcal{C}_k(\langle \phi(I)\rangle)^c.\] To prove (\ref{ijec}.b), let $r'\in J^{ce}=\mathcal{C}_k(\langle J^c\rangle)$. Then there exists an $x\in \langle \phi(J^c)\rangle$ such that $r'+x\in \langle \phi(J^c)\rangle.$ Now, $x\in \langle \phi(J^c)\rangle$ implies $x=\sum_i r'_i\phi(y_i),$ for $r'_i\in R '$ and $y_i\in J^c,$ which subsequently implies $\phi(y_i)\in J$ and hence $x\in J$. With a similar argument, we have $r'+x\in J$. Since $J$ is a $k$-ideal, this means $r'\in J$, as required. Using (\ref{ijec}.a) and (\ref{ijec}.b), we have the proofs of (\ref{ijec}.c) and (\ref{ijec}.d). 
	
(5) By considering the maps $I\mapsto I^e$ and $J\mapsto J^c$, the proof follows.
	
(6) To obtain (\ref{dri}.a), let $r'\in (I_1\cap I_2)^e=\mathcal{C}_k(\langle \phi(I_1\cap I_2)\rangle).$ This implies that $r'+x\in \langle \phi(I_1\cap I_2)\rangle$ for some $x\in \langle \phi(I_1\cap I_2)\rangle,$ and hence $x=\sum_ir'_i\phi(x_i),$ where $x_i\in I_1\cap I_2$ and $r'_i\in R '$. This shows that $$\phi(x_i)\in \mathcal{C}_k(\langle \phi(I_1)\rangle)\cap \mathcal{C}_k(\langle \phi(I_2)\rangle),$$ and hence $x\in I_1^e\cap I_2^e$. Similarly, we obtain $r'+x\in I_1^e\cap I_2^e.$ Since $I_1^e\cap I_2^e$ is a $k$-ideal (since so are $I_1^e$, $I_2^e$, and hence their intersection), this implies that $r'\in I_1^e\cap I_2^e.$ The proof of (\ref{dri}.b) is similar to (\ref{dri}.a), where the main difference is the use of the homomorphism property: $\phi(I_1I_2)=\phi(I_1)\phi(I_2).$ 
	
To show (\ref{dri}.c), suppose that  $r'\in (I_1:I_2)^e=\mathcal{C}_k(\langle \phi(I_1:I_2)\rangle).$ From the definition of $\mathcal{C}_k$, we have $r'+x\in \langle \phi(I_1:I_2)\rangle$ for some $x\in \langle \phi(I_1:I_2)\rangle$. But this means $x=\sum_ir'_i\phi(r_i),$ where $r'_i\in R '$ and $r_i\in (I_1:I_2).$ From this we obtain $\phi(r_i)\in (\phi(I_1):\phi(I_2)).$ Therefore, \[x\in \left(\mathcal{C}_k(\langle \phi(I_1)\rangle) : \mathcal{C}_k(\langle \phi(I_2)\rangle)\right).\]
Similarly, we can show that $r'+x \in (I_1^e:I_2^e).$ Since by Proposition \ref{icj}, $(I_1^e:I_2^e)$ is a $k$-ideal, we must have $r'\in (I_1^e:I_2^e),$ and proves the claim. 
	
Finally, to have (\ref{dri}.d), following the above-mentioned method, we get an $x\in \langle \phi(\mathcal{R}_k(I))\rangle$, which implies $x=\sum_ir_i'\phi(x_i),$ where $r_i'\in R '$ and $x_i\in \mathcal{R}_k(I)$ and hence $x_i\in P$ for all $P\supseteq I$. This implies $\phi(x_i)\in \phi(P)\supseteq \phi(I)$ for all $P\supseteq I$. The argument of the rest of the proof is similar as above.
	
(7) The proofs of all the identities  are analogous to that of (6).
	
(8) Notice that if $y\in \mathcal{R}_k(J)^c,$ then $y\in P$ for all $P$ such that $P\supseteq J^c$, where $P\in \mathrm{spec}_k(R) .$ Since
\[\mathcal{R}^c_kJ=\left( \bigcap_{J\subseteq Q}Q \right)^c=\bigcap_{J^c\subseteq Q^c}  Q^c,\]
we have $y\in \mathcal{R}^c_kJ$. 
\end{proof}

\begin{cremark}
We conclude the paper with the following remarks:
	
$\bullet$ In the sequel of this paper, in \cite{DG23}, we shall explore the behaviour of various types of $k$-ideals under quotients and localizations of semirings. Moreover, in that paper, we aim to study $k$-ideals in special types of semirings, namely, idempotent semirings, Gel'fand semirings, Boolean semirings,  and lattice-ordered semirings.  
	
$\bullet$ In \cite{G99}, a comprehensive study of modules over a semiring has been done, and also the notion of a subtractive semimodule has been introduced. Similar to the distinguished classes of $k$-ideals that we have considered here, one can introduce the same for subtractive semimodules and study their properties. 
	
$\bullet$ Using the $k$-closure operation, we may endow a subbasic closed-set topology on $\mathcal{I}\!(R) $ and study the topological properties of the corresponding spaces. This work has been initiated in \cite{G23}. However, a similar study can also be done for $h$-closure operations. Since there is a nice relationship between $k$-closure and $h$-closure operations (see \cite{H15}), it will be interesting to find topological connections between the respective spaces.
\end{cremark}


\end{document}